\documentclass[11pt]{amsart}
\usepackage{enumerate}
\usepackage{hyperref}
\usepackage{amsthm,amsfonts,amsmath,amscd,amssymb}
\usepackage{xcolor}
\usepackage{graphicx}
\usepackage{caption}
\usepackage{subcaption}
\let\oldmarginpar\marginpar
\renewcommand\marginpar[1]{\-\oldmarginpar[\raggedleft\footnotesize #1]%
{\raggedright\footnotesize #1}}
\theoremstyle{plain}
\newtheorem{thm}{Theorem}[section]
\newtheorem{lemma}[thm]{Lemma}

\newtheorem{prop}[thm]{Proposition}
\newtheorem{cor}[thm]{Corollary}

\theoremstyle{definition}

\newtheorem{definition}[thm]{Definition}
\newtheorem{remark}[thm]{Remark}

\theoremstyle{remark}

\numberwithin{equation}{section}

\newcommand{\D}{\mathbb{D}}
\newcommand{\N}{\mathbb{N}}
\newcommand{\Z}{\mathbb{Z}}

\newcommand{\R}{\mathbb{R}}
\newcommand{\C}{\mathbb{C}}
\newcommand{\sS}{\mathcal{S}}

\renewcommand{\a}{\alpha}
\newcommand{\la}{\lambda}
\newcommand{\La}{\Lambda}
\newcommand{\p}{\varphi}
\newcommand{\e}{\varepsilon}
\newcommand{\dd}{\partial}

\newcommand{\sse}{\subseteq}

\newcommand{\x}{\times}

\newcommand{\supp}{\operatorname{supp}}

\newcommand{\Op}{\operatorname{Op}}

\newcommand{\Diff}{\operatorname{Diff}}
\newcommand{\Symp}{\operatorname{Symp}}
\newcommand{\Cont}{\operatorname{Cont}}

\newcommand{\TOP}{\operatorname{TOP}}

\newcommand{\wt}{\widetilde}

\newcommand{\id}{\text{id}}
\newcommand{\st}{\text{st}}
\newcommand{\std}{\text{st}}
\newcommand{\ot}{\text{ot}}

\newcommand{\Int}{\operatorname{Int}}
\def\Op{{\mathcal O}{\it p}\,}
\newcommand{\bR}{\mathbb{R}}

\newcommand{\bZ}{\mathbb{Z}}

\newcommand{\scrF}{\EuScript{F}}

\usepackage{amscd,euscript}
\usepackage[frame,cmtip,curve,arrow,matrix,line,graph]{xy}

\usepackage{xcolor}

\begin{document}
\begin{abstract}
We construct a symplectic structure on a disc that admits a compactly supported symplectomorphism which is not smoothly isotopic to the identity. The symplectic structure has an overtwisted concave end; the construction of the symplectomorphism is based on a unitary version of the Milnor--Munkres pairing. En route, we introduce a symplectic analogue of the Gromoll filtration.
\end{abstract}

\title{Symplectomorphisms of exotic discs}

\author{Roger Casals}{\thanks{R.C.~is supported by NSF grant DMS-1608018 and a BBVA Research Fellowship}
\address{Massachusetts Institute of Technology, Department of Mathematics, 77 Massachusetts Avenue Cambridge, MA 02139, United States of America}
\email{casals@mit.edu}

\author{Ailsa Keating}\thanks{A.K.~ is partially supported by NSF grant DMS--1505798, by a Junior Fellow award from the Simons Foundation, and by NSF grant DMS-1128155 whilst at the Institute for Advanced Study}
\address{Institute for Advanced Study, Princeton, NJ, 08540, U.S.A.}
\email{keating@ias.edu}

\author{Ivan Smith}\thanks{I.S.~is partially supported by a Fellowship from the EPSRC.  \\ Any opinions, findings and conclusions or recommendations expressed in this material are those of the author(s) and do not necessarily reflect the views of the National Science Foundation.}
\address{Centre for Mathematical Sciences, University of Cambridge, Wilberforce Road, CB3 0WB, United Kingdom}
\email{is200@cam.ac.uk}

\maketitle

\section{Introduction}\label{sec:intro}
In this note, we construct compactly supported symplectomorphisms of certain Euclidean spaces, equipped with non-standard symplectic structures, which are not smoothly isotopic to the identity. 

\begin{thm}\label{thm:main}
Let $\phi\in\pi_0\Diff(D^{4k},\partial)$ be the mapping class of the Kervaire sphere $\Sigma^{4k+1}$. There is a $($non-standard$)$ symplectic structure $\omega_{\text{ot}} \in\Omega^2(D^{4k})$ and a compactly supported symplectomorphism $\p\in\Symp(D^{4k},\partial;\omega_{\text{ot}})$ such that $[\p]=\phi$ in $\pi_0\Diff(D^{4k},\partial)$. 

Therefore,  the inclusion $\Symp(D^{4k},\partial;\omega_{\text{ot}})\sse\Diff(D^{4k},\partial)$ induces a non--zero map
\begin{equation} \label{eqn:comparison}
\pi_0\Symp(D^{4k},\partial;\omega_{\text{ot}})\longrightarrow\pi_0\Diff(D^{4k},\partial)
\end{equation}
whenever $k \notin \{ 1, 3, 7, 15, 31\}$.
\end{thm}

The symplectic 2--form $\omega_{\text{ot}}$ has an overtwisted concave end \cite{ADK05,EM15,Taubes98}, in particular $(D^{4k}, \omega_{\text{ot}})$ is not a Weinstein domain, as we will prove in Proposition \ref{prop:non-Wein}.  The question of whether the analogous map to  \eqref{eqn:comparison} is non-trivial for the standard symplectic structure on the disc is still an open problem, about which we can unfortunately say nothing.

The same techniques used to prove Theorem \ref{thm:main} yield:

\begin{thm}\label{thm:second} Let $(D^{4k-1},\ker\a_\ot)$ be an overtwisted contact structure and $(D^{4k},\omega_\ot)$ its symplectization. Suppose $k \notin \{ 1, 3, 7, 15, 31\}$. 
\begin{itemize}
\item[1.]  We have  $\pi_1 \Cont (D^{4k-1}, \partial; \ker \alpha_{\text{ot}}) \neq \{1\}$.%
\item[2.] 
If $k$ is odd, then $\begin{cases} 
 \pi_j\Symp(D^{4k-j}, \partial; \omega_\ot) \neq \{1\} \ \textrm{for} \  j\in\{2,4\}, \\ \pi_j \Cont (D^{4k-j}, \partial; \ker \alpha_\ot) \neq \{1\} \ \textrm{for} \ j\in \{3,5\}.\end{cases}$
\end{itemize}
\end{thm}
In each case, we find a non-zero element whose image under the composition of the forgetful map to $\pi_i \Diff(D^{4k-i}, \partial)$ with the Gromoll-filtration map to $\pi_0\Diff(D^{4k},\partial)$ 
is the clutching map for the Kervaire sphere. 

The non-trivial classes in both theorems have order at least $2$ and at most $(2k)!$, see Remark \ref{rem:order}. These symplectomorphisms can be implanted into a closed symplectic manifold $(M, \omega)$ by changing $\omega$ near a point $p\in M$ to yield a symplectic structure on $M\backslash \{p\}$ with a concave end, cf. Lemma \ref{lem:inertia}.

The article is organized as follows. In order to establish Theorem \ref{thm:main}, we use the Milnor--Munkres construction of exotic mapping classes in the almost complex setting; this is the content of Section \ref{sec:mmnpairing}. Section \ref{sec:Gromoll} develops the symplectic analogue of the smooth Gromoll filtration, intertwining contact and symplectic structures. Section \ref{sec:proof} contains the proofs of Theorems \ref{thm:main} and \ref{thm:second}. Finally,  Section \ref{sec:app} elaborates on the properties of the symplectic structures featuring in the statements of the above results and provides a few brief remarks on properties of symplectomorphisms (should any exist) for the standard symplectic structure.

\subsection*{Acknowledgements} We are grateful to Diarmuid Crowley, Dusa McDuff and Oscar Randal--Williams for valuable conversations. 


\section{Milnor--Munkres Pairings}\label{sec:mmnpairing}

The group of compactly supported diffeomorphisms of  Euclidean space $\R^{2m}$ is denoted by $\Diff^c(\R^{2m})$. It is equipped with the compact-open topology; its set of connected components $\pi_0\Diff^c(\R^{2m})$ inherits a group structure, which coincides with the group of exotic $(2m+1)$--dimensional spheres under connected sum. Given a mapping class $\eta \in \Diff^c(\R^{2m})$, we denote by $\Sigma_{\eta} \in \Theta_{2m+1}$ the corresponding exotic sphere. 


\subsection{Smooth Milnor-Munkres pairing}
The Milnor--Munkres pairing is, in its simplest form \cite[p.~583]{Lashof}, a group homomorphism
\begin{equation}\label{eqn:milnor_munkres}
\tau: \pi_m SO(m) \times \pi_m SO(m) \longrightarrow \pi_0\Diff^c(\bR^{2m}).
\end{equation}
The map is obtained by a commutator construction. Given a pair of homotopy classes $a,b \in \pi_m SO(m)$, choose two continuous maps
\begin{equation}\label{eqn:matrix_maps}
A,B: (\R^m,\R^m\setminus D^m(1)) \rightarrow (SO(m),\id)
\end{equation}
respectively representing these homotopy classes, and consider the two diffeomorphisms
$$\Phi_A,\Psi_B:\R^m\times\R^m\longrightarrow\R^m\times\R^m,$$
$$\quad\Phi_A: (x,y) \mapsto (x,A(x)(y)),\qquad\Psi_B: (x,y) \mapsto (B(y)(x),y),$$
of the Euclidean space $\bR^{2m}$ endowed with co--ordinates $(x,y)\in\R^m\times\R^m$. These diffeomorphisms are not compactly supported, but their commutator
$$[\Phi_A,\Psi_B]=\Psi_B^{-1}\Phi_A^{-1}\Psi_B\Phi_A$$
is a compactly supported diffeomorphism, and its mapping class depends only on the homotopy classes $a$ and $b$; the pairing \eqref{eqn:milnor_munkres} is then defined by setting $\tau(a,b) = [\Phi_A,\Psi_B]$.

The resulting mapping class $\tau(a,b)\in\pi_0\Diff^c(\bR^{2m})$ defines a smooth structure on the topological sphere $S^{2m+1}$. This smooth structure, not necessarily diffeomorphic to the standard sphere $S^{2m+1}$, also admits a description as the boundary of a smooth plumbing, as follows.

Each homotopy class $a \in \pi_mSO(m)$ defines, by the standard inclusion $SO(m) \rightarrow SO(m+1)$, a homotopy class $\wt a\in\pi_{m} SO(m+1)\cong\pi_{m+1}(BSO(m+1))$ and hence a rank $(m+1)$ vector bundle $\bar{E}_{a}\longrightarrow S^{m+1}$. Explicitly, this vector bundle is obtained by using the element in $\pi_m SO(m+1)$ as the clutching map for the vector bundle trivialised over the two hemispheres of $S^{m+1}$.  Therefore, a pair of classes $a, b $ define a pair of such vector bundles, whose disc bundles we denote by $E_a$ and $E_b$. 
\begin{lemma}\label{Lem:plumbing}
The smooth boundary of the plumbing $E_a \natural E_b$ is diffeomorphic to the exotic smooth $(2m+1)$--sphere defined by $\tau(a, b)$.
\end{lemma}

\begin{proof}
See  for instance \cite[p.~834]{Lawson}
\end{proof}

Consider the smooth $(2m+1)$--dimensional manifold
$$\Sigma=\{z_1^3+z_2^2+\ldots+z_{m+2}^2=1\}\cap\{|z_1|^2+\ldots+|z_{m+2}|^2=1\}\subseteq\C^{m+2},$$
i.e.~ the link of the $A_2$--singularity. The manifold $\Sigma$ is a homotopy sphere,  known as the Kervaire sphere. It relates to the previous discussion via the following;
\begin{cor}\label{Cor:Kervaire_sphere}
Consider two homotopy classes $a,b\in\pi_m SO(m)$ such that
$$\wt a=\wt b=[TS^{m+1}]\in\pi_m SO(m+1).$$
Then $\tau(a,b)$ is diffeomorphic to the Kervaire sphere.
\end{cor}

The class of the tangent bundle $[TS^{m+1}]\in\pi_m  SO(m+1)$ lifts to an element of $\pi_m SO(m)$ when $m$ is even, since odd-dimensional spheres admit nowhere vanishing vector fields; hence the comparison with the classes $a,b\in\pi_m SO(m)$ can be made in a rank $m$ bundle.

Corollary \ref{Cor:Kervaire_sphere} provides the description for the Kervaire sphere we shall use in the proof of Theorem \ref{thm:main}. First, we further examine the case where $m=2n$ is even and the classes $a$, $b$ are in the image of $\pi_{2n} U(n)$, in which situation the Milnor-Munkres maps have nice descriptions as almost-complex maps.


\subsection{Unitary Milnor-Munkres pairings}

Let us start by specifying the definition of an almost complex diffeomorphism.

\begin{definition}\label{defn:almost_cx_diffeo}
A compactly supported almost-complex diffeomorphism of Euclidean space $\bR^{2m}$ is a pair $(f, h)$ consisting of a compactly supported diffeomorphism $f \in \Diff^{c} (\R^{2m})$ and a path $h  = \{h_i \}_{i\in [0,1]}$ of bundle automorphisms $h_i: T\R^{2m} \longrightarrow T\R^{2m}$ such that:
\begin{itemize}
\item[(a)] $h_0 = Df$ and $h_1$ is a $U(m)$--bundle map, i.e.~ the fiber maps
$$(h_i)_p:T_p\R^{2m}\longrightarrow T_{h_i(p)}\R^{2m},\quad\forall p\in\R^{2m},$$
lie in the subgroup $U(m) \leq GL_m(\C)\leq GL_{2m}(\R)$.\\

\item[(b)] Each $h_i$ has compact support: $h_i=\id$ ouside $TK$, for some compact $K\subseteq\R^{2m}$.
\end{itemize}
The collection of such pairs $(f,h)$ is denoted by $\Diff^c (\R^{2m}; J)$. A compactly supported almost--contact diffeomorphism is defined as a stable almost--complex diffeomorphism: a pair $(g,k)$ with $g \in \Diff^c (\R^{2m+1})$ and $\{ k_i \}_{i \in [0,1]}$ a homotopy of bundle maps from the differential $k_0=Dg$ to a $(U(m) \oplus 1)$--bundle map $k_1$ with the obvious compactness conditions.\hfill$\Box$
\end{definition}

The set $\Diff^c(\R^{k};J)$ is topologised as a subspace of $\Diff^c(\R^{k}) \times \textrm{Maps}([0,1],\textrm{End}(T\R^{k}))$.

Implicit in Definition \ref{defn:almost_cx_diffeo} is the choice of the standard (constant) almost complex structure $i$ on $\R^{2m} = \C^m$, via the subgroup of $i$-linear maps $GL_m(\C)$ and its maximal compact subgroup. There is an obvious analogue for a general (not necessarily constant) almost complex structure $J$ on $\R^{2m}$, in which the homotopy $\{h_i\}$ interpolates between $Df$ and a $J$-linear map (or rather, a $J$-linear isometry) through compactly supported bundle automorphisms.  Since the space of almost complex structures on $\R^{2m}$ compatible with the standard orientation is connected, the homotopy type of the resulting space $\Diff^c(\R^{2m};J)$ is independent of $J$, whence the notation. 

Let $\TOP(2m)$ denote the group of homeomorphisms of $\R^{2m}$. A result of  \cite{BL} yields a homotopy equivalence   $\Diff^c (\R^{2m}) \simeq \Omega^{2m+1} (\TOP(2m) / SO(2m))$, and analogously $\Diff^c (\R^{2m}; J) \simeq \Omega^{2m+1} (\TOP(2m) / U(m))$. In particular, the space of almost complex diffeomorphisms is an $h$-space, even if not strictly a group.

\begin{lemma}\label{lem:hit_all}
The forgetful map $\pi_0(j):\pi_0 \Diff^c(\R^{2m}; J) \longrightarrow \pi_0 \Diff^c (\R^{2m})$ is onto for $m\geq3$.
\end{lemma}

\begin{proof}
 The inclusion $U(m) \subset SO(2m)$ induces the following Serre cofibration:
\[
 \frac{SO(2m)}{U(m)} \longrightarrow \frac{\TOP(2m)}{U(m)} \longrightarrow \frac{\TOP(2m)}{SO(2m)}
\]
The associated long exact sequence of homotopy groups gives
\[
\xymatrix{
\ldots \ar[r] & \pi_0 \Diff^c (\R^{2m}; J) \ar[r]^{\text{forget}} & \pi_0 \Diff^c (\R^{2m}) \ar[r]^-\delta & \pi_{2m} (SO(2m) / U(m)) \ar[r] & \ldots
}
\]
where $\delta$ factors through the natural map $\delta': \pi_0 \Diff^c (\R^{2m}) \longrightarrow \pi_{2m} SO(2m)$, induced by pointwise differentiation. By \cite[Proposition 5.4 (iv)]{BL} the map
$$\pi_{2m}(l):\pi_{2m}SO(2m)\longrightarrow \pi_{2m}TOP(2m)$$
induced by the Serre fibration
\[
SO(2m)\stackrel{l}{\longrightarrow} \TOP(2m) \stackrel{p}{\longrightarrow} \TOP(2m) / SO(2m),
\]
is injective and thus $\delta' =  \pi_{2m-1}(p)$ is zero, which yields the required surjectivity.
\end{proof}

\begin{remark}
Fix a symplectic form $\omega$ on $\R^{2m}$. There is a well-defined homotopy class of maps
$$\Symp^c(\R^{2m},\omega)\stackrel{i}{\longrightarrow}\Diff^c (\R^{2m}; J)$$
associated to a choice of compatible almost complex structure $J$ for $\omega$, and a corresponding reduction of the structure group of $(T\R^{2m},\omega)$ to the unitary group. Lemma \ref{lem:hit_all} shows that  in the special case of Euclidean space, the existence of a symplectic lift of a smooth mapping class cannot be obstructed by the lack of existence of an almost-complex lift.

This should be contrasted with a result of Randal-Williams \cite{RW15}, who showed that the corresponding constraint is non-trivial for certain plumbings. In addition, we have recently learnt from D.~Crowley that there is work in progress showing that $\pi_k$--maps are also surjective.\hfill$\Box$
\end{remark}

Suppose now $2m=4n$. There is then a homomorphism
\begin{equation} \label{eqn:unitary_MM}
\xymatrix{
\tau_U: \pi_{2n} U(n) \times \pi_{2n} U(n) \ar[r] & \pi_{2n} SO(2n) \times \pi_{2n} SO(2n) \ar[r]^-{\tau} & \pi_0\Diff^c(\bR^{4n}).
}
\end{equation}
which we refer to as the \emph{unitary Milnor-Munkres pairing}.

\begin{prop}\label{prop:milnormunkres} 
The image of \eqref{eqn:unitary_MM} consists of the class $[\mu] \in \pi_0\Diff^c (\R^{4n})$ of the Kervaire sphere $\Sigma_\mu$ and the identity. In particular,  $\tau_U$ is non--trivial for $n \notin \{ 1, 3, 7, 15, 31\}$.
\end{prop}

\begin{proof}
Since $S^{2n+1}$ admits an almost contact structure, the tangent bundle $TS^{2n+1}$ splits as a trivial real line bundle and an almost complex bundle. It follows that the class $\rho \in \pi_{2n} SO(2n+1)$ of the tangent bundle lifts under the natural maps
\[
\pi_{2n} U(n) \longrightarrow \pi_{2n} SO(2n) \longrightarrow \pi_{2n} SO(2n+1).
\]
Let $\sigma \in \pi_{2n} U(n)$ denote such a lift. Let $A:\bR^{2n}\longrightarrow U(n)$ be a compactly supported map representing the homotopy class $\sigma$ and denote $\mu = [ \Phi_A, \Psi_A]$.
By Corollary \ref{Cor:Kervaire_sphere}, the homotopy sphere $\Sigma_\mu$ is the Kervaire sphere. By work of Browder \cite{Browder} and Hill, Hopkins and Ravenel \cite{HHR}, the $(4n+1)$--dimensional Kervaire sphere $\Sigma \in \Theta_{4n+1}$ is not diffeomorphic to the standard sphere, except when $n=1, 3, 7, 15$, and possibly $31$, which proves the second statement.

One can check using \cite{Kervaire, Harris} that the composition map $g:\pi_{2n}U(n)  \longrightarrow \pi_{2n} SO(2n+1)$ has image contained in a cyclic group $\bZ/2$. Thus the only possibly non-trivial class admitting a lift is the Kervaire sphere $\rho$.
\end{proof}

The argument we use for Theorem \ref{thm:main} requires certain geometric properties of the representatives $A,B$ of Equation \eqref{eqn:matrix_maps}, which we establish in the following proposition. We use the identification $\R^{4n}\setminus\{0\} \cong S^{4n-1} \times (0, + \infty)$ and denote the restriction of a given diffeomorphism $\mu\in\Diff^c(\bR^{4n})$ to the radial spheres by $\mu_t := \mu|_{S^{4n-1} \times \{t \}}$.

\begin{prop}\label{Lem:almostcomplex_MM} \label{prop:milnormunkres2}
A smooth mapping class in the image of \eqref{eqn:unitary_MM} 
has a representative $\mu\in\Diff^c(\bR^{4n})$ such that:
\begin{enumerate}
\item[1.] $\mu$ preserves the distance to the origin,

\item[2.] $\mu$ is supported on the shell $D^{4n}(0.9)\setminus D^{4n}(0.1)\sse\R^{4n}$,

\item[3.] $\mu$ is the identity on the points $(x,y) \in \R^{4n}=\R^{2n}\times\R^{2n}$ s.t.~ $\|x\|< 0.1$ or $\|y\|<0.1$.
\end{enumerate}

In addition, there exists a path of bundle maps $h_i:T\R^{4n}\longrightarrow T\R^{4n}$, $i \in [0,1]$, which covers the diffeomorphism $\mu$ such that:
\begin{itemize}
\item[\emph{I}.] $h_0 = D \mu$, $h_1$ is a $U(2n)$--bundle map, and with the same support $\supp(h_i)=\supp(\mu)$.
\item[\emph{II}.] For $t \in (0,1]$, the bundle maps $h_i$ induce an isotopy of almost-contact forms between $\mu_t^\ast \alpha_{\text{st}}$ 
and the standard contact form $\alpha_{\text{st}}$ on the sphere $S^{4n-1}$. 
\end{itemize}
Note that Property I lifts $\mu$ to an almost-complex map. 
\end{prop}

\begin{proof}
Given two homotopy classes $a, b \in \pi_{2n}U(n)$ represented by compactly supported maps $A, B:\bR^{2n}\longrightarrow U(n)$, we denote $\mu = [ \Phi_A, \Psi_B] \in \Diff^c(\R^{4n})$ as before. By construction, the diffeomorphisms $\Phi_A$ and $\Psi_B$ both preserve the distance to the origin and thus $\mu$ does also. Moreover, we can choose two representatives $A,B$ such that $A(q) = B(q) = \id$ for $q\in D^{2n}(0.1)$, and shrink their respective supports to a thickened sphere, ensuring the second and third properties in the statement. 

Now we want to exhibit a path of compactly supported bundle maps from the differential $D([\Phi_A, \Psi_B]):T\R^{4n}\longrightarrow T\R^{4n}$ to a $U(2n)$--bundle map. 
First, notice that 
\[ 
\big( D\Phi_A \big)_{(x,y)} = \begin{pmatrix}  \id & \ast \\ 0 & \iota(A(x))  \end{pmatrix}
\]
where $\iota:U(2n)\longrightarrow GL(4n,\R)$ is the standard inclusion and thus there is a path $(D_{A,i})_{i\in[0,1]}$ of bundle maps, $i \in [0,1]$, obtained by covering the fixed map $\Phi_A$ on the base and, on the fibres, given by linearly interpolating between the differential $(D\Phi_A)_{(x,y)}$ and the unitary matrix
\[ 
\big( D_{A,1} \big)_{(x,y)} = \begin{pmatrix}  \id &0 \\ 0 & \iota(A(x) ) \end{pmatrix}.
\]
Let us denote the analogous path of bundle maps for $D\Psi_B$ by $(D_{B,i})_{i\in[0,1]}$, and note that, by considering their inverse, these induce paths $(D_{A,i}^{-1})_{i\in[0,1]}$ and $(D_{B,i}^{-1})_{i\in[0,1]}$ of bundle maps for the diffeomorphisms $\Phi_A^{-1}$ and $\Psi_B^{-1}$. By using the chain rule to describe $D([\Phi_A, \Psi_B])$ and applying these four isotopies of bundle maps simultaneously we obtain a path
\[
h_i := D^{-1}_{B,i} \circ D^{-1}_{A,i} \circ D_{B,i} \circ D_{A,i},\quad i\in[0,1],
\]
of compactly supported bundle maps, all covering $[\Phi_A, \Psi_B]$), and interpolating between $D([\Phi_A, \Psi_B])$ and the $U(2n)$--bundle map 
$D^{-1}_{B,1} \circ D^{-1}_{A,1} \circ D_{B,1} \circ D_{A,1}$, as desired.

It thus remains to discuss Property II, for which we consider the radial vector field $\partial_t$. Let us say that a bundle map $D:T\R^{4n}\longrightarrow\R^{4n}$ satisfies $(\dagger)$ if for all points $p\in\R^{4n}$ it has the following two properties
\begin{itemize}
\item[-] $D (T ( S^{4n-1} \times \{ t \})) \subset T \R^{4n}$ coincides with $T ( S^{4n-1} \times \{ t \})$,

\item[-] $D_p (\partial_t) = \partial_t + u_p$  for a tangent vector $u_p \in T ( S^{4n-1} \times \{ \|p\| \})$. 
\end{itemize}
 On the one hand, $D\Phi_A$ and $D\Psi_B$ satisfy $( \dagger )$, as $\Phi_A$ and $\Psi_B$ preserve the distance to the origin. On the other hand, by construction, $D_{A,1}$ and $D_{B,1}$ satisfy $( \dagger )$ as well: in fact, $D_{A,1} (\partial_t) = \partial_t$, and similarly for $B$. Thus the interpolations $D_{A,i}$ and $D_{B,i}$ satisfy $( \dagger )$, as do their inverses. Since the composition of two bundle maps satisfying $( \dagger )$ also satisfies $( \dagger )$, it follows  that $h_i$ satisfies $( \dagger )$ for all $i\in[0,1]$, as required.
\end{proof}

\subsection{Towards Gromoll lifts of unitary Milnor-Munkres maps} In this section we elaborate on the construction described in Proposition \ref{prop:milnormunkres2} by achieving symmetries in further directions than the radial one. These additional symmetries enter in the proof of Theorem \ref{thm:second}, where Proposition \ref{prop:higher_dim_almost_contact} is used.

\begin{lemma} \label{Lem:tangent_lifts}
The class $\sigma\in\pi_{2n}U(n)$ lifts to a class in $\pi_{2n} U(n-1)$ if and only if $n$ is odd.
\end{lemma}

\begin{proof}
As noted in the proof of Proposition \ref{prop:milnormunkres}, the class of the tangent bundle $[TS^{2n+1}] \in \pi_{2n}SO(2n+1)$ is an element of order 2, which admits a lift $\sigma$ to $\pi_{2n}U(n)$.  
For $n=2m+1$ odd, the following exact sequence constructed by Kervaire \cite[p.164]{Kervaire}
\[
0 \longrightarrow \bZ/2 \longrightarrow \pi_{4m+2}U(2m) \longrightarrow \pi_{4m+2}U(2m+1) \longrightarrow 0
\]
yields the claim in this case. In the even case $k=2m$, the corresponding exact sequence is
\[
0 \longrightarrow \pi_{4m} U(2m-1) \longrightarrow \pi_{4m}U(2m) \longrightarrow \bZ/2 \longrightarrow 0.
\]
Thus the classes which admit lifts to $\pi_{4m} U(2m-1)$ are exactly the even multiples of the generator $c$ of the group $\pi_{4m} U(2m) = \bZ/(2m!)$.  However, the classes which map to $TS^{2n+1}$ are exactly the odd multiples of the generator $c$ since the tangent bundle has order two. 
\end{proof}

Lemma \ref{Lem:tangent_lifts} can now be used to prove an analogue of Proposition \ref{prop:milnormunkres2}. In the statement we shall use the co--ordinates $(x,y,z_1,z_2)\in\C^{2n}$, where the pairs are given by $(x,y)\in\C\times\C$ and $(z_1,z_2)\in\times\C^{n-1}\times\C^{n-1}$, and we also denote $z=(z_1,z_2)\in\times\C^{2n-2}$. We also identify
\[
\C \times (\C^{2n-1}\setminus\{0\}) \cong \C \times S^{4n-3} \times (0, +\infty)
\qquad
\C^2 \times (\C^{2n-2}\setminus\{0\}) \cong \C^2 \times S^{4n-5} \times (0, +\infty)
\]
and denote restrictions by $\nu_{x;t} := \nu_{ \{ x \} \times S^{4n-3} \times \{ t\} }$ and $\nu_{x,y;t} := \nu_{ \{ (x, y) \} \times S^{4n-5} \times \{ t\} }$.

\begin{prop}\label{prop:higher_dim_almost_contact}
Let $n\in\N$ be odd. Then there exists a diffeomorphism $\nu \in \Diff^c (\C^{2n})$, whose homotopy class is that of the Kervaire sphere, such that:
\begin{enumerate}
\item[1'.] There are maps $\nu_{x,y}:\C^{2n-2}\longrightarrow\C^{2n-2}$ preserving the distance to the origin such that
\[
\nu (x,y,z) = (x , y, \nu_{x,y}(z)),
\]

\item[2'.] The support satisfies $\supp(\nu)\sse\{ (x,y,z)\in\C^{2n}: \|(x,y)\|< 0.9,\, 0.1 < \|z\| < 0.9 \}$,\\

\item[3'.] $\nu(x,y,z_1,z_2)=\id$ in a region where $\|z_1\|< 0.1$ or $\|z_2\| < 0.1$.
\end{enumerate}

In addition, there exists a path of bundle maps $h_i:T\R^{4n}\longrightarrow T\R^{4n}$, $i \in [0,1]$, which covers the diffeomorphism $\nu$ and satisfies: 
\begin{itemize}
\item[\emph{I}.] $h_0 = D \mu$, $h_1$ is a $U(2n)$--bundle map, and with the same support $\supp(h_i)=\supp(\mu)$.\\
\item[\emph{II'}.] For $t \in (0,1]$, and $x \in \C$, resp.~$(x,y) \in \C^2$, the bundle maps $h_i$ induce an isotopy of almost-contact forms between $\nu^\ast_{x;t} \alpha_\text{st}$, resp.~$\nu^\ast_{x,y;t} \alpha_\text{st}$, and the standard contact form $\alpha_\text{st}$ on $S^{4n-3}$, resp.~$S^{4n-5}$. 
\end{itemize}

\end{prop}

\begin{proof}
First, rearrange the coordinates to $(x,z_1, y,z_2)\in\C^{2n}$. By Lemma \ref{Lem:tangent_lifts}, there exists a representative $A: \C^n \longrightarrow \text{Im}(U(n-1)) \subset U(n)$ of the homotopy class $[TS^{2n+1}]$, where the inclusion $U(n-1) \subset U(n)$ is given by using the final $(n-1)$ co--ordinates. Then the commutator
\[
[ \Phi_A, \Psi_A]
\]
yields a map $\nu$ which satisfies Property 1'. Properties 2' and 3' can be achieved by further taylor-picking the representative $A(x,z_1)$ as follows. By thickening the values $A(x,0)$, we can assume that for fixed $x$ and sufficiently small $z_1$, the diffeomorphism $A(x,z_1)$ is constant. Now the values $A(x, \mathbf{0})$ determine a class in $\pi_2 U(n-1)$ which is zero if $n\geq2$. Thus, after a further homotopy we can assume that $A(x,z_1)=\id$ for $\|z_1\| < 0.1$, which ensures Property 3' and the lower bound in Property 2'. The upper bounds in Property 2' can be achieved by shrinking the domain of $A$. 

For Properties I and II', we will use the same homotopy as in the proof of Proposition \ref{Lem:almostcomplex_MM}, which we still denote by $(h_i)_{i\in[0,1]}$. Property I is satisfied by construction, and we now discuss Property II for the family $\nu_{x; t}$.
By construction, we have the following form for the differential
\[ 
\big( D\Phi_A \big)_{(x, z_1,y, z_2)} = 
\begin{pmatrix}  
1 & 0 & 0 &  \ast  \\
0 & \id & 0 & \ast \\
0& 0 & 1 & \ast \\
0 &  0 & 0 &  \iota(A(x, z_1))  \end{pmatrix}.
\]
Consider the vector field $\partial_t$, where $t$ denotes the distance to the $x$--plane, and in the same vein as before let us introduce the following condition $( \dagger)$:
\begin{itemize}
\item[-] $D T ( \{ x \} \times S^{4n-3} \times \{ t \}) = D T ( \{ x \} \times S^{4n-3} \times \{ t \})$,

\item[-] $D  (\partial_t) = \partial_t + v_p$ for some family of horizontal vectors $v_p \in T ( \{ x \} \times S^{4n-3} \times \{ t \})$.
\end{itemize}

Since $\Phi_A$ and $\Psi_A$ preserve the coordinate $t$, and fix the $x$--coordinate of every point, the bundle maps $D \Phi_A$ and $D \Psi_A$ satisfy $(\dagger)$. In addition the maps $D_{A,1}$ and $D_{B,1}$, defined in the proof of Proposition \ref{Lem:almostcomplex_MM}, also satisfy $(\dagger)$ by construction and thus we can conclude the proof in a completely analogous manner to that of Proposition \ref{Lem:almostcomplex_MM}.
\end{proof}


\section{A symplectic and contact  Gromoll filtration}\label{sec:Gromoll}


The Gromoll filtration \cite{Gromoll} is the subgroup filtration of the group $\pi_0(\Diff^c(\R^n))$ induced by the Gromoll morphisms
$$\la_{k,l}:\pi_k\Diff^c(\R^n)\longrightarrow\pi_{k-l}\Diff^c(\R^{n+l})$$
which are the maps of homotopy groups induced by the natural morphisms
$$\Omega_s^k\Diff^c(\R^{n})\longrightarrow\Omega_s^{k-l}\Diff^c(\R^{n+l}),$$
where $\Omega_s$ denotes the space of smooth loops. The aim of this section is to intertwine this fibration from smooth topology with contact and symplectic structures, the resulting filtration being the content of Proposition \ref{prop:symplectic_Gromoll}.

In its simplest instance, the Gromoll map
$$\la_{k,1}:\pi_k\Diff^c(\R^n)\longrightarrow\pi_{k-1}\Diff^c(\R^{n+1})$$
is the suspension of a loop of diffeomorphisms, and the maps $\la_{k,l}$ for higher values $l\in\N$ can be understood as concatenations of the maps $\la_{k,1}$. We accordingly focus on the contact and symplectic analogues of $\la_{k,1}$ in Propositions \ref{Lem:Moser} and \ref{Lem:Gray}.

\subsection{Suspending a loop of contactomorphisms}\label{ssec:suspendloop-cont}

Let $(M,\ker\alpha)$ be a contact manifold, possibly with boundary, and let us consider
\[
\{\eta_s\}_{s\in[0,1]}\in \Omega^1 \Cont(M,\partial;\ker\alpha)
\] a loop of contactomorphisms such that $\eta_s=\id$ for $s\in\Op(\{0\}\cup\{1\})$. The underlying loop of diffeomorphisms yields a compactly supported diffeomorphism of $M\times\R$ via
\[
\widetilde{\eta} (x,t) = (\eta_t (x), t),
\]
where $t$ is the coordinate on $\R$ and we extend $\eta_s = \id$ in the region $s \notin [0,1]$. Consider the symplectization
$$(M\times\R,\omega)=(M \times \R,d(e^t \alpha)).$$
We would like to upgrade the diffeomorphism $\wt\eta\in\Diff^c(M\times\R)$ to a compactly supported symplectomorphism of the symplectization. 

\begin{prop}\label{Lem:Moser}
Let $(M,\ker\alpha)$ be a contact manifold and $\{\eta_s\}_{s\in[0,1]}\in\Cont(M,\partial;\ker\alpha)$ a loop of compactly supported contactomorphisms. There is a compactly supported exact symplectomorphism $\phi$ of $(M \times \R,d(e^t \alpha))$ which represents the mapping class $[\widetilde{\eta}]\in\pi_0\Diff^c (M \times \R)$.
\end{prop}

The proof of Proposition \ref{Lem:Moser} uses the following technical lemma, with the same input.

\begin{lemma}\label{Lem:exact-contactos}
There exist a compactly supported isotopy $\{\wt\eta_s\}_{s\in[0,1]}$ and a compactly supported smooth function $k: M \times \R\longrightarrow \R$ such that 
\begin{equation}\label{eq:exactcontacto}
\wt\eta_0=\wt\eta,\quad (\wt\eta_1)^\ast \left(  e^t \alpha \right) = e^t \left( \alpha + k(x,t)dt\right).
\end{equation}
\end{lemma} 

\begin{proof}
For each $s\in[0,1]$, $\eta_s$ is a compactly supported contactomorphism and thus there exist compactly supported functions $f_s: M \longrightarrow \R$ such that
\[
\eta_s^\ast (\alpha) = e^{f_s (x)} \alpha. \] 
By definition of $\wt\eta$, the pull--back of the Liouville form $e^t\alpha$ reads
\[
\widetilde{\eta}^\ast \left(   e^t \alpha \right) = e^t \left( e^{f_t(x)} \alpha + g(x,t) dt \right)
\]
where $g: M \times \R \longrightarrow \R$ is a compactly supported smooth function, since $\wt\eta$ is the identity away from a compact set. In order to correct the term introduced by the conformal factors $\{f_t\}$, consider the smooth map
\[
\check{\eta}(x,t) = (\eta_t (x), t-f_t(x) ).
\]
By construction, 
\[
\check{\eta}^\ast \left( e^t \alpha   \right) = e^{t- f_t(x)} \left( e^{ f_t(x)} \alpha + g_1(x,t) dt \right) 
= e^t \alpha + g_2(x,t) dt 
\]
where $g_1,g_2:M\times\R\longrightarrow\R$ are compactly supported smooth functions, for the conformal factors $\{f_t\}$ and $\eta_t$ respectively vanishing and equal the identity away from a compact set. The smooth map $\check{\eta}$ satisfies the Equation \ref{eq:exactcontacto} in the statement as long as $\check{\eta}$ is indeed a diffeomorphism. Surjectivity follows from the fact that each $\eta_t$ is a diffeomorphism, and for any $p  \in M$, the function 
\[
t - f_t (\eta^{-1}_t (p) ) 
\]
is continuous, and agrees with $t$ outside a compact set. It remains to ensure injectivity.
 
Injectivity for $\check{\eta}$ means that there do not exist pairs $(x,t),(y,l) \in M\times\R$ such that 
\[
\eta_t (x) = \eta_l (y) \quad \text{and} \quad  t - f_t(x) = l - f_l (y). 
\]
Equivalently, at no point $p \in M$ do there exist two levels $t, l \in \R$ such that
\begin{equation}\label{eq:intermediate-value}
t - f_t (\eta_t^{-1} (p)) = l - f_l ( \eta_l^{-1} (p)).
\end{equation}
In order to prove this, consider for each point $p \in M$, the smooth function
$$F_p: \R\longrightarrow\R,\quad F_p(t) = f_t (\eta_t^{-1} (p)).$$
By the intermediate value theorem, the equality (\ref{eq:intermediate-value}) above implies that $\check{\eta}$ will be injective if $\|DF_p\| < 1$ for all $p \in M$; note that a priori, we only know that the derivatives $\|DF_p\|$ are bounded. To complete the proof, we use a rescaling trick. 

Fix some small $\epsilon>0$ and define $\rho \in \Diff^c(M \times \R)$ by 
\[
\rho(x,t) = (\eta_{\epsilon t} (x), t).
\]
By construction, 
\[
{\rho}^\ast \left(   e^t \alpha \right) = e^t \left( e^{f_{\epsilon t} (x)} \alpha + \zeta(x,t) dt \right)
\]
for some smooth function $\zeta:M\times\R\longrightarrow\R$ and it suffices to show that the function
\[
\check{\rho}(x,t) := (\eta_{\epsilon t} (x), t-f_{\epsilon t}(x) )
\]
is injective. The analogue of equation (\ref{eq:intermediate-value}) is now
\begin{equation*}
t - f_{\epsilon t} (\eta_{\epsilon t}^{-1} (p)) = l - f_{\epsilon l} ( \eta_{\epsilon l}^{-1} (p)).
\end{equation*}
and the analogue of the function $F_p$ is
\[
G_p (t) := f_{\epsilon t} (\eta^{-1}_{\epsilon t} (p) ) = F_p (\epsilon t).
\]
To ensure injectivity, it suffices to have $\|DG_p\| = \epsilon \|DF_p\| <1$ for all $p \in M$, which can be achieved so long as $\epsilon>0$ is sufficiently small. Suppose we have chosen such an epsilon.

Finally, we need to check that $\check{\rho}$ is isotopic to $\tilde{\eta}$ through compactly supported diffeomorphisms. Note that $\tilde{\eta}$ is isotopic to $\rho$ through compactly supported diffeomorphisms, and we can also consider the linear interpolation
\[
\check{\rho}_l (x,t) = (\eta_{\e t}(x), t - l \cdot f_{\e t}(x)) \quad l \in [0,1]
\]
between the diffeomorphisms $\rho$ and $\check{\rho}$. As before, to show that each $(\check{\rho}_l)_{l\in[0,1]}$ is a diffeomorphism, it suffices to check injectivity. Proceeding as before we get the condition $\|l DG_p \| < 1$ for all $p \in M$, which holds for $l \in [0,1]$. 
\end{proof}

\begin{proof}[Proof of Proposition \ref{Lem:Moser}]
Let us start with the map $\psi=\wt\eta_1$ given to us by Lemma \ref{Lem:exact-contactos}; we will post-compose it with a compactly supported Moser isotopy in order to obtain a compactly supported symplectomorphism of $(M \times \R,d(e^t\alpha))$. First, non-degeneracy of the symplectic 2--form $\omega=d(e^t\alpha)$ gives the pointwise inequality
\begin{equation}\label{eq:non-degenerate1}
\left( d \left( e^t \alpha \right) \right) ^{\wedge n} > 0.
\end{equation}
Consider the pullback of $\omega$ by the diffeomorphism $\psi$
\[
\psi^\ast \left( d \left( e^t \alpha \right) \right) = d \left( e^t \alpha \right) + d(k(x,t)) \wedge dt
\]
where $k:M\times\R\longrightarrow\R$ is a compactly supported smooth function. This pull--back form is a symplectic structure on $M \times \R$, so we also have the pointwise inequality
\begin{equation}\label{eq:non-degenerate2}
\psi^\ast \left( d \left( e^t \alpha \right) \right) ^{\wedge n}  = 
\left( d \left( e^t \alpha \right) \right)^{\wedge n}  
+ C   \left( d \left( e^t \alpha \right) \right)^{\wedge (n-1)} \wedge dk \wedge dt \quad  > 0 
\end{equation}
for some binomial coefficient $C$. Now consider the linear interpolation between these two symplectic forms:
\[
\omega_l := (1-l) \omega + l\psi^*\omega =  d \left( e^t \alpha \right) + l d(k(x,t)) \wedge dt,\quad l\in[0,1].\]
These are closed 2--forms by linearity of the differential, and we also have
\[
(\omega_l)^{\wedge n} = 
 \left( d \left( e^t \alpha \right) \right)^{\wedge n}  
+ l \cdot C   \left( d \left( e^t \alpha \right) \right)^{\wedge n-1} \wedge dk \wedge dt 
\]
which, by equations (\ref{eq:non-degenerate1}) and (\ref{eq:non-degenerate2}), is strictly positive at every point. This implies that each of the 2--forms $\omega_l$ is a symplectic structure, and further they are all exact and agree with $\psi^*\omega$ outside a compact subset of $M \times \R$. Applying the Moser isotopy theorem to this family of symplectic forms $\omega_l$ provides the symplectomorphism $\phi$, as required.
\end{proof}

Proposition \ref{Lem:Moser} constructs the contact--symplectic Gromoll map
$$\la_{1,1}^c:\pi_1\Cont(M,\partial;\ker\a)\longrightarrow\pi_0\Symp(M\times\R,\partial;d(e^t\alpha)),\quad \la_{1,1}^c(\eta)=\phi$$
We now proceed to establish the symplectic--contact counterpart.


\subsection{Suspending a loop of symplectomorphisms}\label{ssec:suspendloop-symp}
Let $(M^{2n}, d\theta)$ be an exact symplectic manifold and denote by
\[
\Symp^c (M, \partial;\theta)
\]
the group of symplectomorphisms $\psi:(M,d\theta)\longrightarrow (M,d\theta)$ such that
\begin{itemize}
\item[-] $\psi$ has compact support and in the interior of $M$, 
\item[-] $\psi$ is an exact symplectomorphism: $\psi^\ast (\theta) = \theta + df$, some smooth function $f: M \longrightarrow \R$ with compact support in $\Int(M)$. 
\end{itemize}

Let $[\{ \phi_s \} ] \in \pi_1 (\Symp^c (M, \partial; \omega, \theta))$ be a path of such exact symplectomorphisms, represented by a one--parameter family of maps $(\phi_s)_{s \in [0,1]}$ which satisfies
\begin{itemize}
\item[-] $\phi_s = \id$ for $s \in \Op (\{0\} \cup \{ 1\} )$;
\item[-] $\phi_s^\ast \theta = \theta + d f_s$, for a smooth family $f_s: M\longrightarrow\R$ with compact support inside $\Int(M)$. 
\end{itemize}
Now consider the contact manifold $(M \times \R, \ker(\theta - dz))$, where $z$ is the coordinate on $\R$. The class $[ \{ \phi_s \} ]$ induces the isotopy class of diffeomorphisms
$$[\wt\phi ] \in \pi_0 \Diff^c (M\times \R),\qquad \wt\phi (x, z) = (\phi_z(x), z)$$
where we have extended the family $\phi_z$ by the identity in the natural manner. 

In order to define the symplectic--contact Gromoll map
$$\la_{1,1}^s:\pi_1\Symp(M,\partial;\theta)\longrightarrow\pi_0\Cont(M\times\R,\partial;\ker(\theta-dz)),$$
we now prove the following proposition.

\begin{prop}\label{Lem:Gray}
There is a contactomorphism $\eta \in \Cont^c (M \times \R, \partial; \ker( \theta- dz) )$ smoothly isotopic to $\wt\phi$ through compactly supported diffeomorphisms of $M \times \R$.
\end{prop}

\begin{proof} First, note that the pull--back of the contact form can be written as 
\[
\phi^\ast (\theta - dz) = \theta + d_x (f_z(x)) + g(x,z)dz - dz
\]
for some smooth function $g: M \times \R\longrightarrow \R$, which is supported in the union of the sets $\supp(\phi_z) \times [0,1]$ for $z\in[0,1]$. Now, let us fix a small constant $\e\in\R^+$ and consider the map
\[
\psi (x,z) := (\phi_{\e z} (x), z).
\]
The maps $\phi$ and $\psi$ are certainly isotopic through compactly supported diffeomorphisms fixing an open neighborhood $\Op (\partial (M \times \R))$. 
Let $e\in\Diff(M\times\R)$ be the diffeomorphism $e(x,z) := (x, \e z)$, which we can use to write $\psi = e^{-1} \circ \phi \circ e$, and thus the chain rule implies
\[
\psi^\ast (\theta - dz) = \theta + d_x (f_{\e z} (x)) + \e g (x, \e z) dz - dz.
\]
Consider the family of one-forms
\[
 \lambda_s := \theta + s \cdot \left( d_x (f_{\e z} (x)) + \e g (x, \e z) dz\right) - dz , \quad s \in [0,1]. 
\]
By construction, $\lambda_0 = \theta-dz$ and $\lambda_1 = \psi^\ast (\theta - dz)$, and we claim that the 1--forms $\lambda_s$ are contact for all $s\in[0,1]$ provided that $\e$ is suitably small.

Indeed, let $f, \, f_\e: M \times \R \longrightarrow \R$ be given by $f(x,z) = f_z (x)$, and $f_e (x,z) = f_{\e z} (x)$, and let $g_\e (x,z) = g (x, \e z)$. Now, for a fixed choice of metric, each of the terms in
\begin{equation} \label{eq:gray}
(d \lambda_0)^n \wedge \lambda_0 - (d \lambda_s)^n \wedge \lambda_s 
\end{equation}
is bounded above in absolute value by a product of binomial coefficients, multiples of $s$, and at least one multiple of one of the following terms:
\begin{eqnarray*}
 \|d_z d_x f_\e\| & = & \e \| d_z d_x f \| \\
 \|\e g_\e\| & = & \e \|g\|  \\
\| d_x ( \e g_\e )\| & =  & \e \|d_x g \|
\end{eqnarray*}
In consequence, for sufficiently small $\e$, the two 1--forms $(d \lambda_0)^n \wedge \lambda_0$ and $(d \lambda_s)^n \wedge \lambda_s$ are of the same non-zero sign at each point, and thus $\lambda_s$ is a contact form for every $s\in[0,1]$. Then, by applying the Gray stability theorem to the family of contact structures $\{ \ker \lambda_s \}_{s \in [0,1]}$ we obtained the desired isotopy and the contactomorphism $\eta$ in the statement. 
\end{proof}

\subsection{Symplectic and contact Gromoll filtration}
By applying Propositions \ref{Lem:Moser} and \ref{Lem:Gray} to $D^k$--parametric families of maps, we have proven the following:

\begin{prop}\label{prop:symplectic_Gromoll}
Let $(M,\theta)$ be an exact symplectic manifold, $(N,\ker\a)$ a contact manifold and $k\in\N$. Then the smooth Gromoll filtration can be refined as follows:

\begin{itemize}
\item[1.] There exists a symplectic--contact Gromoll map
$$\lambda^s_{k,1}: \pi_k \Symp^c (M, \partial; \omega, \theta)\longrightarrow\pi_{k-1} \Cont^c (M \times \R, \partial; \ker (\theta - d z))$$
such that the following diagram commutes:
$$
\xymatrix{
\pi_k \Symp^c (M, \partial; \omega, \theta) \ar[r]^-{\lambda^s_{k,1}} \ar[d] & \pi_{k-1} \Cont^c (M \times \R, \partial; \ker (\theta - d z)) \ar[d]
\\
\pi_k \Diff^c (M, \partial) \ar[r]^-{\lambda_{k,1}} &  \pi_{k-1} \Diff^c (M \times \R, \partial)
}
$$
where the vertical maps are induced by the natural inclusions. \newline

\item[2.] There exists a contact--symplectic Gromoll map
$$\la_{1,1}^c:\pi_1\Cont^c(M,\partial;\ker\a)\longrightarrow\pi_0\Symp^c(M\times\R,\partial;e^t\alpha)$$
such that the following diagram commutes:
\[
\xymatrix{
\pi_k \Cont^c (N, \partial; \ker (\alpha)) \ar[r]^-{\lambda^c_{k,1}} \ar[d] & \pi_{k-1} \Symp^c (N \times \R, \partial; e^t \alpha) \ar[d]
\\
\pi_k \Diff^c (N, \partial) \ar[r]^-{\lambda_{k,1}} &  \pi_{k-1} \Diff^c (N \times \R, \partial)
}
\]
where the vertical maps are induced by the natural inclusions.\hfill$\Box$
\end{itemize}
\end{prop}

Composing the contact and symplectic Gromoll maps alternately, one obtains: 
\begin{itemize}
\item[(a)] For an odd number $2l+1\in\N$,
 $$\lambda_{k,2l+1}^{c}:\pi_k\Cont^c(N,\partial;\ker\a)\longrightarrow\pi_{k-2l-1}\Symp^c(M\times\R^{2l+1},\partial;\theta(\a)),$$

where $\theta(\a)$ denotes the Liouville stabilization of the contact form $\a$, and 
 $$\lambda_{k,2l+1}^{s}:\pi_k\Symp^c(M,\partial;\theta)\longrightarrow\pi_{k-2l-1}\Cont^c(N\times\R^{2l+1},\partial;\a(\theta)),$$

where $\a(\theta)$ denotes the contact stabilization of the Liouville form $\theta$.\\

\item[(b)] For an even number $2l\in\N$,
 $$\lambda_{k,2l}^{c}:\pi_k\Cont^c(N,\partial;\ker\a)\longrightarrow\pi_{k-2l}\Cont^c(N\times\R^{2l},\partial;\wt\a),$$

 where $\wt\a$ denotes the contact stabilization of the contact form $\a$, and 
 $$\lambda_{k,2l}^{s}:\pi_k\Symp^c(M,\partial;\theta)\longrightarrow\pi_{k-2l}\Symp^c(M\times\R^{2l},\partial;\wt\theta),$$

where $\wt\theta$ denotes the Liouville stabilization of the Liouville form $\theta$.\\
\end{itemize}

\begin{remark}
Given a loop of contactomorphisms $\{ \eta_t \}$, the scaling argument in Proposition \ref{Lem:Moser} suggests the following question: is the minimun length in $\R$ of the image of the support of a symplectic representative of $\tilde{\eta}$ an interesting invariant? The methods of the proof yield a naive such length of at most $\max_{ (p,t) \in M\times\R} \left( -DF_p|_t \right)$ for each path, and zero in the case of a loop of strict contactomorphisms. 

By analysing the terms of Equation \ref{eq:gray} in the proof of Proposition \ref{Lem:Gray} more carefully, one gets analogous bounds involving the correction functions $f_t$, where $\phi^\ast_t \theta = \theta + df_t$. In more generality, one could ask about the minimal {\it volume} that can be achieved by representatives of a class in the groups $\pi_k \Symp$ and $\pi_k \Cont$.\hfill$\Box$
\end{remark}


\section{Proof of Theorem \ref{thm:main}}\label{sec:proof}  Let us give the geometric construction underlying the proof of Theorem \ref{thm:main} in a nutshell.

We start with an almost complex diffeomorphism of  $\R^{4n}$ representing the smooth mapping class of the Kervaire sphere, which by Proposition \ref{prop:milnormunkres2} can be assumed to preserve the distance to the origin and act as the identity in a neighborhood of the origin and infinity. Moreover, the associated  loop of diffeomorphisms of the spheres $S^{4n-1}$ is realised by a loop of almost-contact diffeomorphisms. We next show there is an overtwisted contact structure on the sphere $S^{4n-1}$ such that this loop of almost-contact diffeomorphisms is realised by a loop of contactomorphisms. We then 
upgrade this loop of contactomorphisms to a symplectomorphism of the symplectization using Proposition \ref{Lem:Moser}.
 
\begin{remark}
The resulting symplectic structure is non--standard but,  as we shall further discuss in Section \ref{sec:app}, it has appeared in the symplectic topology literature before.\hfill$\Box$
\end{remark}

\subsection{Loop of contactomorphisms}\label{ssec:loop}

Let us focus on the first step. Consider the almost contact structure $(S^{4n-1},J_\st)$ induced by the restriction $J_\st|_{S^{4n-1}}$ of the standard almost complex structure on $S^{4n-1}\times[0.1,0.9]\sse D^{4n}$. By Proposition \ref{prop:milnormunkres2}, there exists an almost complex diffeomorphism $\mu\in\Diff(D^{4n},\partial;J_\st)$ such that
\begin{itemize}
\item[a.] $[\mu] \in\pi_0\Diff(D^{4n},\partial)$ is the clutching map for the Kervaire sphere.
\item[b.] $\mu(S^{4n-1}\times\{ t \})=S^{4n-1}\times\{ t \},\forall t \in(0,1).$
\item[c.] $\mu|_{\Op(\{0\})}=\id$ and $\mu_t:=\mu|_{S^{4n-1}\times\{ t \}}$ is compactly supported away from the disks
$$\Delta\times\{ t \}\sse S^{4n-1}\times\{ t \},$$
where $\Delta\cong D^{4n-1} \sse S^{4n-1}$ is a fixed small disk independent of $t \in(0,1)$.
\end{itemize}

Moreover, by Property II in Proposition \ref{prop:milnormunkres2} each $\mu_t$ is an almost contactomorphism; more precisely, there exists a smooth $(s,t)$--parametric family of almost-contact structures $\xi'_{t,s}$ satisfying
\[
\xi'_{t,0} = \xi_{\st}; \quad \xi'_{t,1} = (\mu_t)_\ast \xi_{\st};
\quad \xi'_{t,s} = \xi_{\st} \text{ for all } t \in \Op(\{0\} \cup \{1 \}).
\]
The maps $\mu_t$ belong to the compactly supported subgroup $\Diff(D^{4n-1},\partial;J_\st)\sse\Diff(S^{4n-1};J_\st)$ by the above properties, where $D^{4n-1} = S^{4n-1} \backslash \Delta$, and satisfy $\mu_t=\id$ for $t=(0,0.1]\cup[0.9,1)$. Examining Property II in Proposition \ref{prop:milnormunkres2}, we see that for all $t$ and $s$, 
\[
\xi'_{t,s}|_{\Delta} = \xi_{\st}|_{\Delta }.
\]

Thus the the maps $\{ \mu_t \}$ together with the data of the family $\xi'_{t,s}$ define a homotopy class $[\mu_t]\in\pi_1\Diff(D^{4n-1},\partial;J_\st)$ of loops of almost contact maps.

Now consider a slightly larger disc embedding $D^{4n-1} \subset S^{4n-1}$, where we now assume we picked an embedding and a metric such that $D^{4n-1}$ has radius one, and 
\[
\cup_{t \in[0,1]}\supp(\mu_t) \subset D^{4n-1}(0.9) \quad \text{and} \quad
D^{4n-1} \setminus D^{4n-1}(0.9) \subset \Delta.
\]
Equip $D^{4n-1}$ with the unique overtwisted contact structure $\xi_\ot$ which is standard on the neighbourhood $\Op(\partial D^{4n-1})$ and lies in the same almost contact class as the structure induced by $J_\st$. In addition, choose the contact structure such that the shell $D^{4n-1}(0.95)\setminus D^{4n-1}(0.9)$
contains an overtwisted disc. In this case, the loop of contact structures $(\mu_t)_*(\xi_\ot)$ consists of overtwisted contact structures sharing a fixed embedded overtwisted disk in the shell region
$D^{4n-1}(0.95)\setminus D^{4n-1}(0.9)$
since the almost contactomorphisms $\mu_t$ are supported away from the overtwisted disc. 
Inserting overtwisted discs in $D^{4n-1}(0.95)\setminus D^{4n-1}(0.9)$, the two-parameter family of almost-contact structures $\xi'_{t,s}$ can be modified to a family $\xi''_{t,s}$ such that:
\[
\xi''_{t,0} = \xi_{\ot}; \quad \xi''_{t,1} = (\mu_t)_\ast \xi_{\ot};
\quad \xi''_{s,t} = \xi_{\ot}\, \, \forall t \in \Op(\{0\} \cup \{1 \}); 
\quad \xi''_{t,s}|_{\Delta \cap D^{4n-1}} = \xi_{\ot}|_{\Delta \cap D^{4n-1}}.
\]

By \cite[Theorem 1.2]{BEM}, applied relative to a fixed neighbourhood $\Op(\partial D^{4n-1})$, there exists a smooth, two-parameter family of contact structures $\{\xi_{t,s}\}_{s\in[0,1]}$ such that
for all $t$,
$$\xi_{t,0}=\xi_\ot; \quad \xi_{t,1}=(\mu_t)_*(\xi_\ot); \quad \xi_{s,t} = \xi_{\ot} \, \, \forall t \in \Op(\{0\} \cup \{1 \}).$$

Note that in general the homotopy must be non--trivial in a neighbourhood of the overtwisted disk and thus in the region $D^{4n-1}(0.95)\setminus D^{4n-1}(0.9)$, but it will be constant on a neighbourhood of the boundary: that is, for all $t$ and $s$ we have
$$
\xi_{t,s}|_{\Op (\partial D^{4n-1})} = \xi_{\ot}|_{\Op (\partial D^{4n-1})} = \xi_{\st}|_{\Op (\partial D^{4n-1})}.$$

For each fixed $t\in[0,1]$, the isotopy of contact structures produces, by using Gray's stability theorem, a path of compactly supported diffeomorphisms $\{g_{t,s}\}_{s\in[0,1]}$ of $D^{4n-1}$ such that
$$(g_{t,s})_*\xi_{t,s}=\xi_\ot,\quad  g_{t,s}|_{\Op(\partial D^{4n-1})}=\id,\quad \forall (t,s) \in [0,1]^2,\mbox{ and }g_{t,s} = \id \quad \forall t \in \Op(\{0\} \cup \{ 1\} )$$
In particular, we obtain the two equalities
$$
g_{t,0}=\id,\quad (g_{t,1}\circ\mu_t)_*\xi_\ot=\xi_\ot,\quad \forall t\in[0,1],
$$
and thence $G_t=\{g_{t,1}\circ\mu_t\}_{t\in[0,1]}$ defines a path of compactly supported contactomorphisms for the contact structure $(D^{4n-1},\partial;\xi_\ot)$, and a homotopy class
$$
[G_t]\in\pi_1\Cont(D^{4n-1},\partial;\xi_\ot)\sse \pi_1\Cont(S^{4n-1};\xi_\ot).
$$
Observe that the path $\{G_t\}$ is smoothly isotopic to $\{\mu_t\}$ because $g_{t,1}$ is the time 1--flow of a vector field, and thus $[G_t]=[\mu_t]\in\pi_1\Diff(D^{4n-1},\partial;J_\st)$ maps to the class of the Kervaire sphere in $\pi_0\Diff(D^{4n},\partial;J_\st)$. This establishes the core of the argument. 

\begin{proof}[Proof of Theorem \ref{thm:main}] By applying Proposition \ref{Lem:Moser} to the loop of contactomorphisms $\{G_t\}_{t\in[0,1]}$ constructed in the previous subsection and the symplectization of the overtwisted contact manifold $(D^{4n-1},\xi_\ot)$ we obtain the statement of Theorem \ref{thm:main}. 
\end{proof}

\begin{remark} 
The Gromoll map $\la_{1,1}:\pi_1 \Diff(D^{2n-1}, \partial) \longrightarrow \pi_0 \Diff(D^{2n}, \partial)$ is surjective. Fix a class $[f] \in \pi_0 \Diff(D^{2n}, \partial)$ and a lift  $[ \{ f_t \} ] \in \pi_1 \Diff(D^{2n-1}, \partial)$. Then if $[ \{ f_t \} ]$ lies in the image of the forgetful map $\pi_1 \Diff(D^{2n-1}, \partial ;J) \longrightarrow \pi_1 \Diff(D^{2n-1}, \partial)$, one can apply the arguments in this section to upgrade $[ \{ f_t \} ]$ to a path $[ \{ \tilde{f}_t \} ] \in \pi_1 \Cont (D^{2n-1}, \partial; \xi_{ot})$, and in turn a representative for $f$ in $\Symp(D^{2n}, \partial; d(e^t \alpha_{ot}))$. We remark that for any class in $\ker (\pi_1 \Diff(\D^{2n-1}, \partial; J) \to \pi_0 \Diff (\D^{2n}, \partial))$, our construction yields a smoothly trivial symplectomorphism which may or may not be symplectically trivial (or even trivial as an almost complex map).
\end{remark}

\begin{remark}\label{rem:order}
Our construction associates a compactly supported symplectomorphism $f_A$ to any element of $\pi_{2n} U(n) \cong \Z/(2n)!$, say with representative $A: \R^{2n} \to U(n)$. Set $A^r (x) = (A(x))^r$. One can check that $f_{A^r}$ is Hamiltonian isotopic to $(f_A)^r$. (One strategy is to deform $A^r$ to a representative given by $r$ copies of $A$ on $r$ disjoint balls in the domain, and follow the steps of the above construction.) On the other hand, picking a null-homotopy from $A^{(2n)!}$ to the identity and following the above steps, one can now see that $f_A^{(2n)!}$ is Hamiltonian isotopic to the identity. (Formally, one would use parametric versions of e.g.~Proposition \ref{prop:milnormunkres2}.) Therefore, the map of Theorem \ref{thm:main} has order at most $(2n)!$ in $\pi_0\Symp(D^{4k},\partial;\omega_{\textrm{ot}})$.
\end{remark}

\subsection{$3$- and $5$-dimensional families of contactomorphisms}

Following the argument in the previous Subsection \ref{ssec:loop}, starting from the 3 and 5--dimensional families of almost contactomorphisms of Proposition \ref{prop:higher_dim_almost_contact}, we obtain the following result:

\begin{prop}\label{prop:higher_htpy_groups_contacto}
For $n \geq 3$ odd, there are classes
\[
[H_t ] \in \pi_3 \Cont (D^{4n-3}, \partial; \xi_\ot) \quad \text{and} \quad 
[K_t ] \in \pi_5 \Cont (D^{4n-6}, \partial; \xi_\ot) 
\]
such that under the composition
\[
\pi_3 \Cont (D^{4n-3}, \partial; \xi_\ot) \longrightarrow \pi_3 \Diff(D^{4n-3}, \partial ) \longrightarrow \pi_0 \Diff(D^{4n}, \partial),
\]
where the first is induced by inclusion, and the second is a Gromoll map,
the class $[H_t]$ maps to the clutching map for the Kervaire sphere, and similarly for $[K_t]$. In particular, for any odd $n$ such that $n \notin \{ 1,3,7,15,31 \}$, the homotopy groups 
\[ \pi_3 \Cont (D^{4n-3}, \partial; \xi_\ot) \quad \text{and} \quad 
 \pi_5 \Cont (D^{4n-5}, \partial; \xi_\ot)
\]
are non-trivial.\hfill$\Box$
\end{prop}

An immediate consequence of Propositions \ref{prop:higher_htpy_groups_contacto} and \ref{prop:symplectic_Gromoll} is the following:

\begin{cor}
Consider $(D^{2n}, \omega_\ot)$, the symplectization of the overtwisted contact manifold $(D^{2n-1}, \ker \alpha_\ot)$.
For all odd $n$ with $n \notin \{ 1,3,7,15,31 \}$, the homotopy groups
\[
\pi_2 \Symp(D^{4n-2}, \partial; \omega_\ot) \quad \text{and} \quad
\pi_4 \Symp(D^{4n-4}, \partial; \omega_\ot) 
\]
are non-trivial.\hfill$\Box$
\end{cor}

Browder \cite{Browder-hspace}   proved  that any $h$-space with non-trivial second homotopy group does not have the homotopy type of a finite cell complex, and Hubbuck \cite{Hubbuck} proved that any homotopy-commutative $h$-space which is homotopy equivalent to a finite cell complex has vanishing homotopy groups in all degrees $\geq 2$. 

\begin{cor}
For all odd $n$ with $n \notin \{ 1,3,7,15,31 \}$, each of the spaces 
$$\Symp(D^{4n-2}, \partial; \omega_\ot),\qquad \Cont(D^{4n-3}, \partial; \xi_\ot)$$
$$\Symp(D^{4n-4}, \partial; \omega_\ot),\qquad \Cont(D^{4n-5}, \partial; \xi_\ot)$$
does not have the homotopy type of a finite-dimensional cell complex.
\end{cor}

\section{Concluding Remarks}\label{sec:app}

This section collects some supplementary material. First, we discuss the symplectic structure obtained by symplectizing an overtwisted contact structure. Then, we globalize the construction in the previous section by implementing it inside a general symplectic cobordism. 
Finally, we mention some facets of the problem in relation to the standard symplectic structure on Euclidean space. 

\subsection{Overtwisted Symplectizations}\label{ssec:sympot}

Recall that an exact symplectic manifold $(X,\omega = d\theta)$ is  \emph{Weinstein} if it admits a (complete) Liouville vector field $Z$, $\mathcal{L}_Z(\omega) = \omega$, which is gradient-like for an exhausting Morse function on $X$. 

\begin{prop}\label{prop:non-Wein}
Let $(\R^{2n-1},\xi_\ot)$ be an overtwisted contact structure, $\sS(\R^{2n-1},\xi_\ot)$ its symplectization and $n\geq3$. Then $\sS(\R^{2n-1},\xi_\ot)$ does not support a Weinstein structure.
\end{prop}

\begin{proof}
In a symplectization, any compact subset can be Hamiltonian displaced from itself.  On the other hand, in a Weinstein manifold a closed exact Lagrangian submanifold is never Hamiltonian displaceable since its self-Floer cohomology is well-defined and non-vanishing. It therefore suffices to construct a closed exact Lagrangian in $\sS(\R^{2n-1},\xi_\ot)$.

Consider the Legendrian unknot $\Lambda_0\sse(\R^{2n-1},\ker(e^1\a_\ot))$ at the contact level of unit height, and note that in the concave piece of the symplectization $\{t\leq1\}\sse\sS(\Op(\La_0),\xi_\std)$ of a Darboux neighborhood $(\Op(\La_0),\xi_\std)$ of this Legendrian $\Lambda_0$ there exists an embedded exact Lagrangian disk $L_-=D_0$ which bounds the Legendrian unknot $\La_0$. Simultaneously, the contact structure $(\R^{2n-1},\ker(e^1\a_\ot))$ is overtwisted and thus the Legendrian unknot $\La_0$ is also a loose Legendrian \cite{BEM,CMP15}. The existence $h$--principle for exact Lagrangian embeddings with concave Legendrian boundary \cite{EM13} now implies that there exists a exact Lagrangian $L_+\sse\{t\geq1\}\sse(\R^{2n-1},\ker(e^1\a_\ot))$ with boundary $\La_0$. This constructs an exact Lagrangian embedding $L=L_-\cup_{\partial\La_0}L_+$ inside the symplectization of any overtwisted contact structure.
\end{proof}

\subsection{Globalisation to symplectic cobordisms}\label{ssec:globalising}

The construction of symplectic structures with symplectic exotic mapping classes detailed in Section \ref{sec:proof} can be implanted in a local manner into the concave end of a $2n$--dimensional symplectic cobordism $(X,\omega)$. Indeed, it suffices to use the following Weinstein cobordism $(M,\la,f)$ which interpolates, as a smooth concordance, between an overtwisted contact structure $(S^{2n-1},\xi_\ot)$ in the concave end and the standard contact structure $(S^{2n-1},\xi_\std)$.

\begin{prop}[\cite{CMP15}] 
Suppose that $n\geq3$.  Then there is a Weinstein structure $(M,\lambda, f)$ on the smoothly trivial cobordism $M \cong [0,1]\x S^{2n-1}$ such that $(\dd_+M,\lambda) \cong (S^{2n-1},\xi_\std)$ and $(\dd_-M,\ker(\lambda))$ is the unique overtwisted contact sphere in the almost contact class of $\xi_\std$.
\end{prop}

This Weinstein cobordism $(M,\la,f)$ can be implanted in any symplectic cobordism $(X,\omega)$ by performing a vertical connected sum with a piece of the symplectization of the non--empty concave end $(\partial_-X,\la_-)$. For a closed symplectic manifold $(\wt X,\omega)$, corresponding to the case where the concave end is empty, we can remove a Darboux ball and obtain a symplectic cobordism $(X,\omega)$ whose concave end $(\partial_-X,\la_-)=(\partial_ X,\la_-)$ is contactomorphic to the standard contact sphere $(S^{2n-1},\xi_\std)$. Then, the Weinstein cobordism $(M,\la,f)$ can be concatenated and yields a symplectic structure
$$(X,\omega_\ot):=(M,\la,f)\cup_{(S^{2n-1},\xi_\std)} ((\wt X,\omega)\setminus (D^{2n},\la_\st))$$
with a conical singularity at the concave end $(\partial_-M,\la)$.

These symplectic structures $(X,\omega_\ot)$ have a unique concave overtwisted end or, equivalently, a conical symplectic singularity modelled on an overtwisted sphere.  Such conical symplectic structures have appeared in symplectic topology before:  they play an essential role in the $h$--principle for symplectic cobordisms \cite{EM15}, since the $h$--principle fails unless the singularities are allowed \cite{Gr85, McDuff};  and overtwisted conical ends are the model for the singularities of near--symplectic structures \cite{ADK05,Taubes98}.

Consider the map 
$$i^c:\Diff^c(M)\longrightarrow\Diff^c(X)$$
induced by the  inclusion $i:(M,\la,f)\longrightarrow (X,\omega_\ot)$. The diffeomorphisms $f\in\Diff^c(M)$ constructed in Section \ref{sec:proof} have non-trivial image in  $\pi_0(i^c)([f])\in\Diff^c(X)$ precisely when the Kervaire sphere (is smoothly exotic and) does not lie in the inertia group of $X\times S^1$.

\begin{lemma} \label{lem:inertia}
Let $(X,\omega)=(\Sigma_1\times\cdots\times\Sigma_n,\omega_1\oplus\cdots\oplus\omega_n)$ be the product of compact symplectic surfaces $(\Sigma_i,\omega_i)$, $1\leq i\leq n$, each one of arbitrary genus. The inertia group $I(X\times S^1)$ vanishes.
\end{lemma}
\begin{proof} The inertia group $I(X\times S^1)$ equals the group of smooth mapping classes on $X$ which are supported in a disk and pseudo--isotopic to the identity \cite[Proposition 1]{Levine}. Consequently, $I(X\times S^1)$ is contained in the inertia group of any manifold containing $X$ in codimension 1 \cite[Theorem 4.1]{Ge16}. Thus $I(X\times S^1)\subseteq I(S^{2n+1})=0$, thanks to the embedding $X\subseteq S^{2n+1}$. (When each $\Sigma_i$ has genus at most 1, the result was known from \cite{Schultz}.) \end{proof}

In particular, we obtain smoothly non-trivial symplectomorphisms of ``punctured" symplectic structures on tori and products of 2-spheres.

\subsection{The standard symplectic structure}\label{ssec:stdsymp}

A natural question is whether one can use the Milnor-Munkres description of the clutching map of the Kervaire sphere to find a representative for it that is a symplectomorphism for the standard symplectic form; this remains open.

There exist  representatives for the generator of $\pi_{2n} U(n)$ with large amounts of symmetry, e.g.~coming from Samelson products \cite{Bott}; explicit formulae are  given in \cite{PR}. Before launching herself into calculations, the curious reader should note that for these representatives we have checked that the linear interpolation between the standard  symplectic form and its pullback is not a path of symplectic forms.

We conclude with three remarks, whose proofs we only outline, given that they pertain to non-trivial symplectomorphisms of $(D^{2k},\omega_{\std})$ which are not known to exist.

\begin{remark} Let $\phi \in \Symp (D^{2k},\partial;\omega_{\std})$.
\begin{enumerate}
\item There is a well-defined canonically $\Z$-graded Floer cohomology group $HF^*(\phi)$, see \cite{Seidel:more_vm, McLean, Uljarevic}.  We claim this is necessarily isomorphic to $HF^*(\id)$, hence of rank 1 and concentrated in degree zero. Indeed, one can implant the graph of $\phi$ into the zero-section of $T^*S^{2k}$ to obtain an exact Lagrangian submanifold $L_{\phi}$ which is Floer-theoretically isomorphic to the zero-section \cite{FSS}, and then argue that $HF^*(\phi)$ appears as a summand in $HF^*(S^{2k},L_{\phi})$.\\


\item If $\phi$ exists, it yields  a non-trivial element in $\pi_0\Symp(T^{2k},\omega_{\std})$, by Lemma \ref{lem:inertia}. On the other hand, from the arguments of \cite[Section 9]{AbouzaidSmith} and Orlov's classification of autoequivalences of derived categories of abelian varieties, one sees that this symplectomorphism acts trivially on the (unobstructed or full) Fukaya category $D^{\pi}\scrF(T^{2k})$. This gives a strong sense in which $\phi$ would be invisible to classical Floer theory.\\

\item If $\phi$ has image equal to the Kervaire sphere under the map $\pi_0\Symp(D^{2k},\partial; \omega_{\std}) \to \pi_0\Diff(D^{2k},\partial)$, and if $k$ is even and $2k+1\neq 2^j-3$, there are counterexamples to the ``nearby Lagrangian conjecture". Indeed, either $L_{\phi} \subset T^*S^{2k}$ provides a counterexample, or, by using a suspention of a Hamiltonian isotopy from $L_{\phi}$ to the zero-section, one can construct a Lagrangian embedding $\Sigma_{[\phi]\circ u^2} \hookrightarrow T^*S^{2k+1}$ for some $u\in \Diff(D^{2k},\partial)$ (compare to \cite{Evans-Rizell}; the unknown reparametrization map $u$ arises from the fact that the isotopy to the zero-section need not be one of parametrized Lagrangians).  The dimension constraints on $k$ imply \cite[Theorem 1.1]{Brumfiel} that the Kervaire sphere has no square root in $\Theta_{2k+1}$, hence $\Sigma_{[\phi]\circ u^2}$ is exotic. This connects the existence question considered in this paper to the nearby Lagrangian conjecture, which has seen much recent activity.
\end{enumerate}
\end{remark}

\bibliography{bib}{}
\bibliographystyle{plain}
\end{document}